\documentclass[11pt,reqno]{amsart}
\usepackage{graphicx} 
\newcommand{\Q}{\mathbb{Q}}
\newcommand{\Z}{\mathbb{Z}}
\newcommand{\K}{\mathbb{K}}
\newcommand{\R}{\mathbb{R}}
\newcommand{\C}{\mathbb{C}}
\newcommand{\N}{\mathbb{N}}

\renewcommand{\S}{\mathbb{S}}

\renewcommand{\P}{\mathbb{P}}
\renewcommand{\a}{\mathfrak{a}}
\newcommand{\p}{\mathfrak{p}}
\newcommand{\one}{\mathbbm{1}}

\usepackage[colorlinks=true, linkcolor=blue, citecolor=red, urlcolor=blue]{hyperref}
\usepackage{bbm}
\usepackage{comment}
\usepackage{todonotes}
\usepackage{amsmath}
\usepackage[capitalize]{cleveref}
\usepackage{fullpage}
\usepackage{amsfonts}
\usepackage{amsthm}
\usepackage{amsmath,amssymb}
\newtheorem{theorem}{Theorem}[section]
\newtheorem{definition}[theorem]{Definition}
\newtheorem{proposition}[theorem]{Proposition}

\newtheorem{corollary}[theorem]{Corollary}
\newtheorem{question}{Question}
\newtheorem{lemma}[theorem]{Lemma}
\theoremstyle{definition}
\newtheorem{remark}[theorem]{Remark}
\crefname{conjecture}{conjecture}{conjectures}
\Crefname{conjecture}{Conjecture}{Conjectures}
\crefname{corollary}{corollary}{Corollary}
\Crefname{corollary}{Corollary}{Corollaries}
\title{Pointwise Ergodic Averages Along the Omega Function in Number Fields}
\author{Diego C\'espedes}
 \address[Diego C\'espedes]{Departamento de Ingenier\'{\i}a Matem\'atica and Centro de Modelamiento Matem{\'a}tico, Universidad de Chile \& IRL 2807 - CNRS, Beauchef 851, Santiago, Chile} \email{diegocespedes@ug.uchile.cl} 
\author{Sebasti\'an Donoso} \address[Sebasti{\'a}n Donoso]{Departamento de Ingenier\'{\i}a Matem\'atica and Centro de Modelamiento Matem{\'a}tico, Universidad de Chile \& IRL 2807 - CNRS, Beauchef 851, Santiago, Chile} \email{sdonosof@uchile.cl}
\thanks{Both authors were partially funded by Centro de Modelamiento Matemático (CMM) FB210005, BASAL funds for centers of excellence from ANID-Chile and by ANID/Fondecyt/1241346}

 
\begin{document}
\begin{abstract}
We study a general criterion for guaranteeing that an ergodic average exhibits the strong sweeping-out property. This result implies, in particular, the failure of pointwise convergence of these averages. Our result applies to averages along the Omega function in number fields, generalizing a result of Loyd. We also show that the averages
\[\frac{1}{N^2}\sum_{1\leq m,n \leq N}f(T^{\Omega(m^2+n^2)}x)\] exhibit the strong sweeping-out property, which answers affirmatively a question posed by Le, Moreira, Sun, and the second author.

On the other hand, using number-theoretic methods, we establish the pointwise convergence of averages along the $\Omega$ function defined on the ideals of a number field in uniquely ergodic systems. Using this dynamical framework, we also derive several natural number-theoretic consequences of independent interest.
\end{abstract}
\maketitle

\section{Introduction}\label{section:introduction}

It is well-known that the prime number theorem can be rephrased in terms of averages of the Liouville function. More precisely, the prime number theorem is equivalent to 
\begin{equation} \label{eq:vanishing_omega_pnt} 
\frac{1}{N} \sum_{k=1}^{N} \lambda(k)=o(1)_{N\to\infty}.
\end{equation}
Here, the Liouville function $\lambda(n)\colon \N\to \{-1,1\}$ is the completely multiplicative function (that is, $\lambda(n\cdot m)=\lambda(n) \cdot \lambda(m)$ for $m,n\in \N$) defined by $\lambda(p)=-1$ for every prime $p$. 
Rewriting $\lambda$ as $\lambda(n)=(-1)^{\Omega(n)}$, where $\Omega(n)$ counts the prime factors of $n$ (counting repeated factors with multiplicity), \eqref{eq:vanishing_omega_pnt} is equivalent to 
\[\frac{1}{N}\sum_{k=1}^{N}(-1)^{\Omega(k)}=o(1)_{N\to \infty},\]
which suggests an ergodic-theoretic approach to the problem. This is indeed what Bergelson and Richter did in their breakthrough work \cite{Bergelson_Richter_dynamical_pnt:2022}, where they proposed and proved the following statement involving dynamical systems. 
\begin{theorem}[{\cite{Bergelson_Richter_dynamical_pnt:2022}}]
Let $(X, T)$ be a uniquely ergodic system and let $\mu$ be its unique invariant measure. Then, for any $f\in C(X)$ and any $x \in X$ 
\[ \lim_{N\to \infty} \frac{1}{N} \sum_{n=1}^{N}f(T^{\Omega(n)}x)=\int_{X}fd\mu.\]
\label{thm:bergelson_richter_dynamical_theorem}
\end{theorem}
Using particular classes or examples of uniquely ergodic systems, such as finite rotations, irrational rotations, certain symbolic systems, and nilsystems, allows us to obtain the prime number theorem and far-reaching number-theoretic generalizations of it (see \cite{Bergelson_Richter_dynamical_pnt:2022} for further discussions). 

It is natural to ask about norm and/or pointwise convergence of averages as in \cref{thm:bergelson_richter_dynamical_theorem} under no unique ergodicity of the system or a non-continuous function $f$. Loyd answered these questions in \cite{Loyd_dynamical_omega:2023}, where she proved that the convergence in norm still holds, but the pointwise convergence fails for an extensive class of dynamical systems.
\begin{theorem}[{\cite{Loyd_dynamical_omega:2023}}]\label{thm:Loyd}
Let $(X,\mathcal{X},T,\mu)$ be an ergodic invertible m.p.s. where the measure $\mu$ is non-atomic ($\mu(\{x\})=0$ for all $x\in X$). Then, for every $\varepsilon>0$, there exists a measurable set $A$ such that $\mu(A)<\varepsilon$ and for $\mu$-almost every $x\in X$, we have 
\[\limsup_{N\to \infty}\frac{1}{N} \sum_{n=1}^{N}\one_{A}(T^{\Omega(n)}x) =1 \quad \text{ and } \quad \liminf_{N\to \infty}\frac{1}{N} \sum_{n=1}^{N}\one_{A}(T^{\Omega(n)}x) =0.\]
\end{theorem}
This type of behavior is commonly referred to in the literature as the \textit{strong sweeping-out property}. See \Cref{subsection:strong_sweeping} for the definition.

It is worth noting that recently, Loyd and Mondal \cite{loyd_mondal_omega:2025} studied the behavior of ergodic averages along the Omega function using other weights. For instance, they showed that by choosing \emph{logarithmic weights}, the averages also exhibited strong sweeping-out behavior.

Interestingly, when considering {\emph{double-logarithmic averages}}, the averages do not exhibit sweeping-out behavior, but they converge pointwise.
\subsection{Extensions to other number fields}
Recently, Donoso, Le, Moreira, and Sun proved in \cite{Donoso_Le_Moreira_Sun_Averages_multiplicative_Gaussians:2024} an analogous result to that of Bergelson and Richter, in the context of Gaussian integers. 
As a consequence of a stronger ``dynamical theorem'' \cite[Theorem E]{Donoso_Le_Moreira_Sun_Averages_multiplicative_Gaussians:2024}, they obtained the convergence along the $\Omega(\cdot)$ function of the norm of Gaussian integers.
\begin{theorem}\cite[Theorem D]{Donoso_Le_Moreira_Sun_Averages_multiplicative_Gaussians:2024}\label{thm:ThmD}
Let $(X, T)$ be a uniquely ergodic system with the unique invariant measure $\mu$. Then for any $x \in X$ and any $f \in C(X)$,
\[
\lim_{N \to \infty} \frac{1}{N^2} \sum_{1\leq m,n\leq N} f\left(T^{\Omega(m^2 + n^2)}x\right) = \int_X f \, d\mu.
\]
\end{theorem}
In the same paper, they raised several questions (for further discussions see \cite[Section 5]{Donoso_Le_Moreira_Sun_Averages_multiplicative_Gaussians:2024}). One of the questions is in the spirit of Loyd’s result about the convergence of the averages in \Cref{thm:ThmD} for general ergodic systems.
\begin{question}\cite[Question 5.2]{Donoso_Le_Moreira_Sun_Averages_multiplicative_Gaussians:2024} \label{question:question_5.2}
Let $(X,\mu,T)$ be a non-atomic ergodic system. Is it true that there exists a measurable set $A \subset X$ such that for $\mu$-almost every $x \in X$,
\[
\limsup_{N\to\infty}\frac{1}{N^2}\sum_{1\leq m, n \leq N} \one_{A}(T^{\Omega(m^2+n^2)}x)
= 1
\]
and
\[
\liminf_{N\to\infty}
\frac{1}{N^2}\sum_{1\leq m, n \leq N} \one_{A}(T^{\Omega(m^2+n^2)}x)
= 0 \, ?
\]
\end{question}
We answer \Cref{question:question_5.2} affirmatively. The method we use to answer this question is more general and yields a sufficient condition for the averages to exhibit the strong sweeping-out property. The criterion can also be applied to generalize Loyd's result (\Cref{thm:Loyd}) to arbitrary number fields, which applies in particular to Gaussian integers; for precise definitions, see Section \ref{section:background}. 
\begin{theorem}\label{thm:pointwise_failure}
Let $ \K $ be a number field, $ \mathcal{O}_\K $ its ring of integers, and $ G_\K $ the set of non-zero ideals of $\mathcal{O}_\K $. Let $(X,\mathcal{X},T, \mu) $ be an ergodic invertible m.p.s. with non-atomic $ \mu $. For either $
c(\mathfrak{a}) = \Omega_\K(\mathfrak{a})$ or $ c(\mathfrak{a}) = \omega_\K(\mathfrak{a})$, for every $\varepsilon>0$
there exists a set $ A \in \mathcal{X} $ such that $\mu(A)<\varepsilon$, and for $\mu$-almost every $ x \in X $,
\[
\limsup_{N \to \infty} \frac{1}{\#\{ \mathfrak{a} \in G_\K : \mathcal{N}(\mathfrak{a}) \leq N \}} \sum_{\substack{\mathfrak{a} \in G_\K \\ \mathcal{N}(\mathfrak{a}) \leq N}} \one_A(T^{c(\mathfrak{a})} x) = 1,
\]
and
\[
\liminf_{N \to \infty} \frac{1}{\#\{ \mathfrak{a} \in G_\K : \mathcal{N}(\mathfrak{a}) \leq N \}} \sum_{\substack{\mathfrak{a} \in G_\K \\ \mathcal{N}(\mathfrak{a}) \leq N}} \one_A(T^{c(\mathfrak{a})} x) = 0.
\]
\end{theorem}
In particular, Loyd’s result (\Cref{thm:Loyd}) corresponds to the case $ \K = \Q $ and $ c = \Omega_{\Q} =\Omega$ in \Cref{thm:pointwise_failure}. 
We can also show that several other averages exhibit the strong sweeping-out property; see \Cref{section:failure_of_pointwise}.

Given that \Cref{thm:Loyd} remains valid when we consider an arbitrary number field, by \Cref{thm:pointwise_failure}, it is natural to ask whether recent results \cite{loyd_mondal_omega:2025} can be extended to this more general setting.

Regarding results where we do have convergence, we propose the following generalization of \cref{thm:bergelson_richter_dynamical_theorem} to an arbitrary number field.
\begin{theorem}\label{thm:dynamical_omega_number_fields} Let $\K$ be a number field, $(X,T)$ a uniquely ergodic system with invariant measure $\mu$, $f \in C(X)$, and $x \in X$. For either $
c(\mathfrak{a}) = \Omega_\K(\mathfrak{a})$ or $ c(\mathfrak{a})= \omega_\K(\mathfrak{a})$, we have
\[ \lim_{N\to \infty}\frac{1}{\#\{ \mathfrak{a} \in G_{\mathbb{K}}: \mathcal{N}(\mathfrak{a})\leq N \}}\sum_{\substack{\mathfrak{a} \in G_\K \\ \mathcal{N}(\mathfrak{a}) \le N}}f(T^{c(\mathfrak{a})}x)=\int_{X}fd\mu .\]
\end{theorem}
In particular, by choosing the number field $\K=\Q$ and $c(\mathfrak{a})=\Omega_{\Q}(\mathfrak{a})$, we recover Bergelson and Richter’s statement. Also, by applying this theorem to the $2$-point rotation, we obtain
\begin{equation} \label{eq:vanishing_omega_number_fields}
\lim_{N\to \infty}\frac{1}{\#\{ \mathfrak{a} \in G_{\mathbb{K}}: \mathcal{N}(\mathfrak{a})\leq N \}}\sum_{\substack{\mathfrak{a} \in G_\K \\ \mathcal{N}(\mathfrak{a}) \le N}}(-1)^{\Omega_\K(\mathfrak{a})}=0,
\end{equation}
which is equivalent to the prime number theorem for number fields, also known as Landau’s prime ideal theorem. It is worth noting that our proof of \Cref{thm:dynamical_omega_number_fields} does not provide a new proof of Landau’s prime ideal theorem, since it relies on a strengthening of that theorem. To see a proof of Landau’s prime ideal theorem, which is based on the ideas of Bergelson and Richter, we refer to \cite{Burgin_new_elementary_landau:2025}. 

We remark that our methods for proving \Cref{thm:dynamical_omega_number_fields} are not enough to prove the convergence for uniquely ergodic systems of other types of ergodic averages. For instance, we are not able to show the convergence of 
\begin{equation}\label{equation:unexplored_averages}
\frac{1}{\#\{ \mathfrak{a} \in G_{\mathbb{K}}: \mathcal{N}(\mathfrak{a} )\leq N \}}\sum_{\substack{\mathfrak{a} \in G_\K \\ \mathcal{N}(\mathfrak{a}) \le N}}f(T^{\Omega(\mathcal{N} (\mathfrak{a}))}x).\end{equation}
The methods of \cite{Donoso_Le_Moreira_Sun_Averages_multiplicative_Gaussians:2024} are able to deal with these and also more types of averages, in the particular case of the Gaussian integers.
We expect that the convergence of the averages as in \eqref{equation:unexplored_averages}, and possibly more general results, should follow from a “dynamical statement” concerning multiplicative actions defined in terms of prime ideals (or prime elements) of a number field. It would be interesting to show such a statement, even in particular classes, such as class number one fields (which are P.I.D.s) or imaginary quadratic fields. This problem remains open, but it is currently under investigation. A positive result in number fields would imply \cite[Theorem A]{Bergelson_Richter_dynamical_pnt:2022}, \cite[Theorem E]{Donoso_Le_Moreira_Sun_Averages_multiplicative_Gaussians:2024}, \Cref{thm:dynamical_omega_number_fields} of this work, and would have many more consequences.
\section*{Acknowledgements}
We thank Pieter Moree for pointing out useful references for the number-theoretic preliminaries. We also thank Sovanlal Mondal for drawing our attention to his joint paper with K. Loyd, and for helpful conversations that led to the full resolution of \Cref{question:question_5.2}. This insight allowed us to improve a partial answer given in an earlier version of this article.
We are also very grateful to the anonymous referee for valuable comments, which helped improve the clarity and presentation of our results.

\section{Preliminaries}\label{section:background}
\subsection{Notation}
 We denote the set of positive integers by $\N$ and the set of (rational) prime numbers by $\P$. Throughout this paper, variables such as $ m$ and $ n$ implicitly denote positive integers unless otherwise stated. For a finite set $A$, we denote its cardinality by $\#A$. For an infinite set $A\subseteq\R$, $|A|$ denotes its Lebesgue measure. Following standard conventions, $\log_2(x)$ denotes $\log(\log(x))$. For functions
$f,g\colon \N\to \R$ we denote $f=O(g)$ if there exists a constant $C$ and an $N_0$ such that for all $N\geq N_0$, $|f(N)|\leq C|g(N)|$, and $f=o_{N\to \infty}(g)$ to denote $\lim_{N\to \infty}\frac{ f(N)}{g(N)}=0$.

We follow the notation of Wu (\cite{Wu_sharpening_Selberg-Delange:1996}). $ \mathbb{K} $ denotes a number field and $\mathcal{O}_{\K}$ denotes its ring of integers. We denote an ideal of $\mathcal{O}_{\mathbb{K}} $, usually by the letter $\mathfrak{a} $, a prime ideal by the letter $\mathfrak{p}$, and $ G_{\mathbb{K}}$ denotes the multiplicative semigroup of non-zero ideals of $\mathcal{O}_{\mathbb{K}} $. The (absolute) norm of an ideal $ \mathfrak{a} $ is denoted by $ \mathcal{N}(\mathfrak{a}) $.
An ideal $ \mathfrak{a} $ can be decomposed in terms of finitely many prime ideals (see, e.g., \cite{Marcus_Number_Fields:1977}), and the decomposition is unique up to the order of the factors. If we write, 
$ \mathfrak{a} = \prod_{i=1}^{\infty} \mathfrak{p}_i^{\alpha_i}$, $\alpha_i\in \N
\cup\{0\}$, \footnote{We adopt this notation of convenience. Note that the product contains only finitely many $i$ such that $\alpha_i>0$.}
we define:
\begin{itemize}
\item $\omega_{\mathbb{K}}(\mathfrak{a}) $ as the number of distinct prime ideals dividing $ \mathfrak{a} $.
\item $ \Omega_{\mathbb{K}}(\mathfrak{a}) $ as the total number of prime ideal factors of $ \mathfrak{a} $, counted with multiplicity, i.e., $\sum \alpha_i $.
\end{itemize}
 $\Omega_{\K}$ satisfies the property $\Omega_{\K}(\mathfrak{a}\cdot \mathfrak{b})=\Omega_{\K}(\mathfrak{a})+\Omega_{\K}(\mathfrak{b})$ for $\mathfrak{a},\mathfrak{b} \in G_\K$. Any function that fulfills this property is called completely additive. Similarly, $\omega_{\K}$ satisfies $\omega_{\K}(\mathfrak{a}\cdot \mathfrak{b})=\omega_{\K}(\mathfrak{a})+\omega_{\K}(\mathfrak{b})$, provided that $\mathfrak{a}$ and $\mathfrak{b}$ do not share a prime factor in their factorizations. Any function satisfying this property is called additive.
We also denote the {\em Dedekind Zeta function} of a number field as the function\[\zeta_{\K}(s) =\sum_{\a\in G_\K}\frac{1}{\mathcal{N}(\mathfrak{a})^s}.\]
This function is initially defined only for $D=\{z\in \C:\Re(z)>1\}$, but admits a meromorphic extension to the complex plane with a single pole at $z=1$, which we still denote $\zeta_\K(\cdot)$. The residue at this pole, denoted by $\rho_\K$, is related to the growth of ideals.
\begin{theorem}[\cite{RamVanOrder}] \label{thm:weber_theorem}Let $\K$ be a number field. Then there exists $0<\delta <1$ such that
\[ \#\{\mathfrak{a} \in G_{\mathbb{K}}: \mathcal{N}(\mathfrak{a}) \leq N \}=\sum_{\substack{\mathfrak{a} \in G_{\mathbb{K}} \\
\mathcal{N}(\mathfrak{a}) \leq N }} 1=\rho_\K N+O(N^{\delta}).\]
\end{theorem}
\subsection{\texorpdfstring{Asymptotic behavior of $\omega_{\K}(\cdot)$ and $\Omega_\K(\cdot)$}{Asymptotic behavior of ω and Ω}}
To prove \Cref{thm:dynamical_omega_number_fields} and \Cref{thm:pointwise_failure}, we require some number-theoretical input. In particular, we need generalizations of the classical Hardy--Ramanujan theorem and of the Sath\'e--Selberg theorem, both of which have extensions to arbitrary number fields. These two results describe statistical properties of $\omega_{\K}(\cdot)$ and $\Omega_{\K}(\cdot)$.
First, we start by presenting a version of the Erd\H{o}s--Kac theorem for number fields, as established by Liu in \cite{Liu_generalization_Erdos-Kac_appl:2004}. While this specific form of the theorem is not explicitly stated in Liu’s work, it can be easily derived by combining Lemma 3 and Theorem 1 from that paper.
\begin{theorem}[Erd\H{o}s--Kac theorem for number fields]\label{thm:erdos_kac}
 Let $\K$ be a number field. For either $
c(\mathfrak{a}) = \Omega_\K(\mathfrak{a})$ or $ c(\mathfrak{a})=\omega_{\K}(\mathfrak{a})$, and for every $\gamma\in \R$
\[
\lim_{x \to \infty,x\in \Q}
\frac{\#\{
\mathfrak{a}\in G_\K : \mathcal{N}(\mathfrak{a}) \leq x,\,
\frac{ c(\mathfrak{a}) - \log_2 x }{ \sqrt{ \log_2 x } }
\leq \gamma
\}}{\#\{ \mathfrak{a}\in G_\K : \mathcal{N}(\mathfrak{a}) \leq x \}}
= G(\gamma),
\]
where $ G(\gamma):=\frac{1}{\sqrt{2\pi}}\int_{-\infty}^\gamma e^{-\frac{1}{2}t^2}dt $ is the standard Gaussian distribution function.
\end{theorem}
As a consequence of \cref{thm:erdos_kac}, we obtain the following version of the Hardy--Ramanujan theorem for number fields.
\begin{theorem}[Hardy--Ramanujan for number fields]\label{thm:hardy_ramanujan} Let $\K$ be a number field. For either $
c(\mathfrak{a}) = \Omega_\K(\mathfrak{a})$ or $ c(\mathfrak{a})=\omega_{\K}(\a)$, for any $\varepsilon>0$ and every sufficiently large $C=C(\varepsilon)>0$, we have
\[ \lim_{x\to \infty,x \in \mathbb{Q}} \frac{ \#\{\mathfrak{a} \in G_{\K}:\mathcal{N}(\mathfrak{a})\leq x,\Bigl|\frac{c(\mathfrak{a})-\log_2(x)}{\sqrt{\log_2(x)}}\Bigr|> C \}}{\#\{\mathfrak{a} \in G_{\K}:\mathcal{N}(\mathfrak{a})\leq x\}}\leq \varepsilon.\] 
\end{theorem}
This means that, for any small given $\varepsilon > 0$, if we choose a sufficiently large $C$, then at least a ($1-\varepsilon$) fraction of the ideals have a number of prime factors in their decomposition (with or without multiplicity) lying in the moving interval
\begin{equation}
I_{N,C} := \Bigl[ \log_2(N) - C\sqrt{\log_2(N)} , \, \log_2(N) + C\sqrt{\log_2(N)}\Bigr]\cap \Z.
\label{definition:interval_n_c}
\end{equation}
For $ k \in I_{N,C} $, the distribution of the functions $ \omega_{\K} (\cdot)$ and $\Omega_{\K}(\cdot) $ can be described precisely. In the case of integers, this was first established by Erd\H{o}s \cite{Erdos_integers_k_prime_factors:1948}, refined by Sath\'e \cite{Sathe1953I, Sathe1953II, Sathe1954III, Sathe1954IV}, and simplified by Selberg \cite{Selberg1954}. This result is now known as the Sath\'e--Selberg theorem. Wu \cite{Wu_sharpening_Selberg-Delange:1996} later extended the Sath\'e--Selberg theorem to arbitrary number fields and provided asymptotic formulas for the following counting functions: $N_k(x) := \# \{ \mathfrak{a} \in G_{\mathbb{K}} : \mathcal{N}(\mathfrak{a}) \leq x, \Omega_{\mathbb{K}}(\mathfrak{a}) = k \}$ and $\pi_k(x) := \# \{ \mathfrak{a} \in G_{\mathbb{K}} : \mathcal{N}(\mathfrak{a}) \leq x, \omega_{\mathbb{K}}(\mathfrak{a})= k \}$. 
Using the explicit form of the error term, together with \Cref{thm:weber_theorem} and Stirling’s approximation as in \cite[Lemma 3.4]{Loyd_dynamical_omega:2023}, these formulas can be expressed in terms of exponential functions, which is convenient for further applications.
\begin{lemma}\label{lemma:approximation_level_sets}
Let $C>0$, and let $I_{N,C}$ be as above. Then, for $k\in I_{N,C}$
\[ \frac{ N_k(N)}{\# \{ \mathfrak{a} \in G_{\mathbb{K}} : \mathcal{N}(\mathfrak{a}) \leq N\}} = \frac{e^{-\frac{1}{2}(\frac{k-\log_2(N)}{\sqrt{\log_2(N)}})^2}}{\sqrt{2\pi \log_2(N)}}(1+\varepsilon_k(N)) ,\]
where the error terms converge uniformly to zero, i.e.,
\[ \lim_{N\to \infty} \sup_{k\in I_{N,C}} |\varepsilon_k(N)|=0.\]
The same asymptotic formula holds for $\pi_k(N)$.\end{lemma}
\subsection{Miscellaneous}
We will also employ a far-reaching generalization of the Erd\H{o}s--Kac theorem for polynomial functions. Note that the original result allows several polynomials with an arbitrary number of variables, but we will only state the case of a single polynomial with two variables.
\begin{corollary}\cite{ElBaz_loughran_sofos_normal:2022} 
\label{thm:general_erdos_kac_polynomials}
Let $g \in \mathbb{Z}[x_1,x_2]$ be a nonconstant polynomial and let $c$ denote the number of irreducible, primitive, non-constant polynomials $f$ such that $f \mid g$.

Let
\[
C_N = \left\{ (x,y) \in \mathbb{Z}^2 : \max \{|x|,|y|\}\leq N,\; g(x,y) \neq 0 \right\}
.\]

Then, for any $\gamma\in \R$
\[\lim_{N\to \infty} \frac{ \# \left\{(x,y)\in C_N: \frac{\omega(g(x,y)) - c \log_2N}
{\sqrt{c\log_2 N}}\leq \gamma\right\}}{\# C_N}=G(\gamma). \]
\end{corollary}
\begin{remark}
In the language of probability theory, this means that when $C_N$ is given the uniform (normalized) measure, the sequence of random variables
\[
C_N \longrightarrow \mathbb{R}, 
\qquad
(x,y) \longmapsto 
\left(
\frac{\omega(g(x,y)) - c \log_2 N}
{\sqrt{c \log_2 N}}
\right)
\]
converge in {\emph{distribution}} to a standard Gaussian. This probabilistic language will be useful in the sequel.

 \end{remark}
We will also borrow a result from probability theory, known as Slutsky's theorem, often used in probabilistic number theory.
\begin{theorem}[{Slutsky's theorem. (See, for example, \cite[Remark 1]{Billingsley:1969})}]\label{thm:slutsky}
Let $(X_N)_{N\in \N}$ and $(Y_N)_{N\in \N}$ be two sequences of random variables such that for every $N$, $X_N$ and $Y_N$ are defined on the same probability space. If $X_N$ converges in distribution to a random variable $G$ and $Y_N$ converges to zero in $L^1$, i.e., $\lim_{N\to \infty}\mathbb{E}[|Y_N|]=0$, then $X_N+Y_N$ also converges in distribution to $G$.
\end{theorem}
\subsection{Measure-preserving systems} We recall some basic notions of dynamical systems. A {\em topological dynamical system} (t.d.s.) is a pair $(X, T)$ with $X$ a compact metric space and $T\colon X \to X$ continuous.

A {\em measure-preserving dynamical system} (m.p.s.) is a quadruple $(X, \mathcal{X}, \mu, T)$, where $(X, \mathcal{X}, \mu)$ is a probability space, $T\colon X\to X$ is measurable and $\mu$ is $T$-invariant (that is, $\mu(T^{-1}(A)) = \mu(A)$ for all $A \in \mathcal{X}$). If $T \colon X \to X $ is an invertible transformation and its inverse is measurable, we say the system is {\em invertible}. A m.p.s. is called {\em ergodic} if for a set $A\in \mathcal{X}$, $A=T^{-1}A$ implies $\mu(A)\in \{0,1\}$.
In all that follows, we may (and will) assume that $X$ is a compact metric space, $\mathcal{X}$ is its Borel $\sigma$-algebra, and $\mu$ is a Borel probability measure. By the Krylov--Bogolyubov theorem \cite{Kryloff_Bogoliouboff_theorie_measure_systeme_dyn:1937}, every t.d.s. has at least one invariant measure. If this measure is unique, the system is called {\em uniquely ergodic}. If for every $k\in \N$ the t.d.s. $(X,T^k)$ is uniquely ergodic, the system is called {\em totally uniquely ergodic}.
\subsection{The strong sweeping-out property}\label{subsection:strong_sweeping}
In order to address questions about ergodic averages that are not necessarily of the form $\frac{1}{N}\sum_{n=1}^{N}f(T^{a_n}x)$ (where $(a_n)_{n\in \N}$ is a sequence of integers), it is helpful to introduce a new definition.
\begin{definition}
Let $(\Phi_N)_{N\in \N}$ be a sequence of nonempty finite sets, and let $\tau:\bigcup_{N\in \N}\Phi_N\to\N\cup\{0\}$ be a function. Then, the pair $((\Phi_N)_{N\in \N},\tau)$ is referred to as an averaging scheme.
\end{definition}
Given an averaging scheme, we can consider the corresponding ergodic averages
\begin{equation} \label{eq:average_general}
\frac{1}{\#\Phi_N}\sum_{n\in \Phi_N} f(T^{\tau(n)}x), \quad N\in\mathbb{N}.
\end{equation}
By taking $\Phi_N = \{1, \ldots, N\}$ and $\tau(n)=n$, we recover the classical ergodic averages. Moreover, by setting 
$\Phi_N = \{\mathfrak{a}\in G_\mathbb{K} : \mathcal{N}(\mathfrak{a}) \le N\}$ 
and letting $\tau(\mathfrak{a}) = \Omega_{\mathbb{K}}(\mathfrak{a})$ or 
$\tau(\mathfrak{a}) = \omega_{\mathbb{K}}(\mathfrak{a})$, 
we recover the averages presented in \Cref{section:introduction}.
This definition will be helpful for generalizing classical definitions and results in ergodic theory to specific averages that appear unnatural when indexed by the natural numbers, but arise naturally from a different underlying structure, such as a number field.

We can ask whether the ergodic averages along an averaging scheme are well-behaved or ill-behaved. For instance, we can ask whether the analog of Birkhoff's theorem holds along the averaging scheme. Strong sweeping-out-behavior corresponds to the case in which the ergodic averages oscillate in the most extreme way for a characteristic function.
\begin{definition}
We say that an averaging scheme $((\Phi_N)_{N\in \N},\tau)$ satisfies the strong sweeping-out property, if for every invertible, ergodic and non-atomic m.p.s. $(X, \mathcal{X},\mu, T)$, we have that for every $\varepsilon>0$, there is a set $E$, such that $\mu(E)<\varepsilon$, and for $\mu$-a.e. $x\in X$
\[\limsup_{N\to \infty} \frac{1}{\#\Phi_N} \sum_{k\in \Phi_N}\one_{E}(T^{\tau(k)}x)=1,\text{ and} \qquad \liminf_{N\to \infty} \frac{1}{\#\Phi_N} \sum_{k\in \Phi_N}\one_{E}(T^{\tau(k)}x)=0. \]
\end{definition}

The strong sweeping-out property has been the subject of research in several papers. The main criterion for proving that an averaging operator, such as the ones in \eqref{eq:average_general}, satisfies this property is the one introduced by del Junco and Rosenblatt in \cite{delJunco_Rosenblatt_counterex_erg_numb_th:1979}. 
\begin{proposition}\cite[Theorem 1.3]{delJunco_Rosenblatt_counterex_erg_numb_th:1979} Let $(X, \mathcal{X}, \mu)$ be a probability space. Let $(T_N)_{N\in \N}$ be a family of operators $T_N:\mathcal{X}\to L^{1}(\mu)$ satisfying mild conditions\footnote{These conditions are: being 1 on the whole space, monotone, linear, and continuous in measure. These conditions are immediately satisfied for the operators we consider. We refer to \cite{delJunco_Rosenblatt_counterex_erg_numb_th:1979} for the precise criteria.} and the following maximal inequality, for every $\varepsilon>0$ and $M_1\in \N$, there is a set $A\in \mathcal{X}$, such that $\mu(A)<\varepsilon$, and 
\begin{equation} \label{eq:del_junco_maximal}
\mu( \{x\in X: \sup_{N\geq M_1 }T_N(A)(x)\geq 1-\varepsilon\})\geq 1-\varepsilon .\end{equation}
Then, the family of operators satisfies the strong sweeping-out property, that is, for every $\varepsilon>0$, there is a set $B\in \mathcal{X}$, such that $\mu(B)<\varepsilon$ and for $\mu$-a.e. $x\in X$ \[\limsup_{N\to \infty}T_N(B)(x)=1\text{, and } \liminf_{N\to \infty}T_N(B)(x)=0.\] 
\end{proposition}
As in positive pointwise convergence results, what typically underlies this is the boundedness of a maximal operator. In the case of quantitative failure of pointwise convergence, such as the strong sweeping-out property, what's behind it is the extreme failure of that boundedness as captured by the maximal inequality \eqref{eq:del_junco_maximal}.

However, verifying the maximal inequality \eqref{eq:del_junco_maximal} directly is not always easy to do. To prove that the maximal inequality holds in our cases of interest, we rely on the following result, which has already been stated in the literature (see, e.g., \cite[Lemma 2.2]{MondalRoyWierdl2023} and \cite[Theorem 2.3]{Akcoglu_et_al:1996}), but we state it in our language. This result tells us that we only need to verify the maximal inequality on a finite interval of the integers, and then we can transfer it to any system in a very large class.
\begin{proposition} \label{thm:transference_principle}
Let $((\Phi_N)_{N\in \N},\tau)$ be an averaging scheme. It induces a function defined on finitely supported functions $f:\Z\to \R$, given by 
\[A_N(f)(x)=\frac{1}{\#\Phi_N} \sum_{k\in \Phi_N}f(x+\tau(k)).\]
Assume it satisfies the following:
Given any $\varepsilon>0$, $M>0$, and $N_0 \in \N$, there is an integer $N_1\geq N_0$, and a finite subset $E\subseteq \Z$ such that
\[ \#\{x\in \Z: \sup_{N_0\leq N\leq N_1}A_N(\one_{E})(x)\geq 1-\varepsilon\} \geq M \# E.\]
Then, the averaging scheme $( (\Phi_N)_{N\in \N},\tau)$ satisfies the strong sweeping-out property.
\end{proposition}
\section{Failure of pointwise convergence}\label{section:failure_of_pointwise}
In this section, we provide a criterion to establish strong sweeping-out for averages. Using this criterion, we extend \Cref{thm:Loyd}, give a positive answer to \Cref{question:question_5.2}, and present many more examples of averages that exhibit sweeping-out behavior. 

Loyd \cite{Loyd_dynamical_omega:2023} proves \Cref{thm:Loyd} in two main steps. In the first step, the author shows that the averages can be approximated by another family of averages, denoted $T_{N,C}$, for which classical tools can be employed to show the failure of pointwise convergence, by employing directly \Cref{eq:del_junco_maximal}. Once this approximation is established, she proves that the strong sweeping-out behavior of $T_{N,C}$ can be transferred to the average of interest.

In the approximation argument, two number-theoretical results are used, one is the Hardy-Ramanujan theorem (\Cref{thm:hardy_ramanujan}), which follows from the Erd\H{o}s--Kac theorem (\Cref{thm:erdos_kac}), and establishes mean properties of the averages along $(\Omega(n))_{n\in \N}$. The second theorem is the Sath\'e--Selberg theorem (see \Cref{lemma:approximation_level_sets}), which gives fine asymptotics for $\frac{\pi_k(N)}{N}$. Note that Sath\'e--Selberg implies the Erd\H{o}s--Kac theorem (see, e.g., \cite[Exercise 218]{Tenenbaum:2015}). Proving the Sath\'e--Selberg theorem requires knowing analytic properties of the Dirichlet series of $\Omega(n)$, while the Erd\H{o}s--Kac theorem requires less machinery. Recently, Loyd and Mondal showed that the Sath\'e--Selberg theorem is not necessary to establish the failure of pointwise convergence along $(\Omega(n))_{n\in \N}$, and that the Erd\H{o}s--Kac theorem is sufficient for this purpose. This is achieved by means of the following result.

\begin{corollary}\label{cor:cor_1.7_loyd_mondal:2025}\cite[Corollary 1.7]{loyd_mondal_omega:2025} Let $(a(n))_{n\in \N}$ be a sequence of integers. Suppose that for all $\varepsilon>0$, there exist increasing subpolynomial functions $(b(n))_{n\in \N}$ and $(p(n))_{n\in \N}$
 satisfying $p(n)=o(b(n))$
 and
 \[ \limsup_{N\to \infty}\frac{\#\left\{1\leq n\leq N: |a(n)-b(n)|>p(n)\right\}}{N}\leq \varepsilon.\]
 Then, the averages $T_Nf(x)=\frac{1}{N}\sum_{n=1}^{N}f(T^{a(n)}x)$ satisfy the strong sweeping-out property.
\end{corollary}
\begin{remark}
This applies to $a(n)=\Omega(n)$ by choosing $b(n)=\log_2(n)$ and $p(n)=C\sqrt{\log_2(n)}$. These are subpolynomial sequences that satisfy the hypothesis of \Cref{cor:cor_1.7_loyd_mondal:2025} in virtue of the Hardy-Ramanujan theorem.
\end{remark}
This suggests that, to establish the strong sweeping-out for more general averages, it suffices to show that they satisfy a Hardy-Ramanujan-type theorem. To show this, it suffices to show they satisfy an Erd\H{o}s--Kac type theorem. We define this property as follows.
\begin{definition} An averaging scheme $((\Phi_N)_{N\in \N},\tau)$ satisfies an Erd\H{o}s--Kac type theorem if there exists a constant $c>0$ such that for every $\gamma\in \R$, we have
\[\lim_{N\to \infty}\frac{\# \left\{k\in \Phi_N: \frac{\tau(k)-c\log_2(N)}{\sqrt{c\log_2(N)}}\leq\gamma\right \}}{\#\Phi_N}=G(\gamma),\]
where $G$ is the standard Gaussian distribution.
\end{definition}
\begin{remark}
If we equip $\Phi_N$ with the uniform probability measure, then the expression above means that the sequence of random variables for $N\in \N$,
\begin{align*} 
X_N:\Phi_N\to &\R\\
k\to &X_N(k)=\frac{\tau(k)-c\log_2(N)}{\sqrt{c\log_2(N)}}
\end{align*} converges to a standard Gaussian \textit{in distribution}. 
\end{remark}
The following theorem generalizes \Cref{cor:cor_1.7_loyd_mondal:2025} of Loyd and Mondal to arbitrary averaging schemes satisfying an Erd\H{o}s--Kac type theorem.
\begin{theorem}\label{thm:general_failure_of_pointwise} If an averaging scheme $((\Phi_N)_{N\in \N},\tau) $ satisfies an Erd\H{o}s--Kac type theorem, then it satisfies the strong sweeping-out property.
\end{theorem}
 Note that we decided to focus on an Erd\H{o}s--Kac type condition rather than more general subpolynomial sequences, since this case is well-studied and sufficient for our purposes. We will illustrate many examples satisfying the conditions of \Cref{thm:general_failure_of_pointwise}, but postpone its proof to the last part of the section.
 
First, in the case of number fields, Liu's result (\Cref{thm:erdos_kac}) implies that both $\omega_{\K}$ and $\Omega_\K$ functions exhibit an Erd\H{o}s--Kac-type growth.
\begin{proof}[Proof of \Cref{thm:pointwise_failure}]
Let $\K$ be a number field. For $N\in \N$, set $\Phi_N=\{\a\in G_\K:\mathcal{N}(\a)\leq N\}$, and $\tau=\Omega_{\K}$ or $\tau=\omega_\K$. The conclusion follows directly by applying \Cref{thm:general_failure_of_pointwise}.
\end{proof}
More generally, in the same paper where these results were established, Liu gives criteria for an abelian monoid to satisfy an Erd\H{o}s--Kac type theorem.
In a related work, Das, Kuo, and Liu \cite{Das_Kuo_Liu_distribution_generators_monoids:2026,Das_Kuo_Liu_erdos_kac_subsets_number_fields:2025,Das_Kuo_Liu_distribution_generalization_erdos_kac:2026} showed that $h$-free and $h$-full ideals of number fields also satisfy an Erd\H{o}s--Kac theorem for both $\omega_\K$ and $\Omega_\K$ functions. Recall that a non-zero ideal $\a$ can be written as $\a=\prod_{i=1}^{\infty}\p_i^{\alpha_i}=\prod_{i=1}^{n}\p_i^{\alpha_i}$, where $\alpha_i>0$ for $i\in \{1,\ldots, n\}$ and $\alpha_i=0$ for all $i >n$. Let $h\in \N$, $h\geq 2$, we say that an ideal $\a$ is $h$-free if $\alpha_i\leq h-1$ for all $i\in \{1, \ldots, n\}$, and say that it is $h$-full, if $\alpha_i\geq h$ for $i\in \{1, \ldots, n\}$. 
These works also extend the results to abelian monoids satisfying certain axioms. We state only one representative application corresponding to the case of $\Omega_\K$ along $h$-full ideals.
\begin{corollary} \label{cor:failure_full} Let $(X, \mathcal{X}, \mu, T)$ be an invertible ergodic m.p.s., where $\mu$ is non-atomic. Let $\K$ be a number field, let $\Phi_N=\{\a \in G_\K:\a \text{ is $h$-full, and } \mathcal{N}(\a)\leq N\}$, then for every $\varepsilon>0$, there is a set $A\in \mathcal{X}$, satisfying $\mu(A)<\varepsilon$ and for $\mu$-a.e. $x\in X$
\[\limsup_{N\to \infty}\frac{1}{\#\Phi_N} \sum_{\substack{\a \text{ is $h$-full} \\ \mathcal{N}(\a)\leq N}} \one_{A}(T^{\Omega_{\K}(\a)}x)=1,\]
 and 
 \[\liminf_{N\to \infty}\frac{1}{\#\Phi_N} \sum_{\substack{\a \text{ is $h$-full} \\ \mathcal{N}(\a)\leq N}} \one_{A}(T^{\Omega_{\K}(\a)}x)=0.\]
\end{corollary}
In the specific case where $\K=\Q$, Li, Wang, Wang, and Yi showed in \cite{LiWangWangYi:2025}, among other things, that for totally uniquely ergodic systems, the following ergodic theorem holds.
\begin{theorem} \label{thm:liwangwangyi_k_full}\cite[Theorem 1.5] {LiWangWangYi:2025} Let $(X,T)$ be a totally uniquely ergodic system with invariant measure $\mu$. For $f\in C(X)$ and $x\in X$, we have 
\[\lim_{N\to \infty}\frac{1}{\#\{1\leq n\leq N : n \text{ is $h$-full}\}}\sum_{\substack{n=1 \\ n \text{ is $h$-full} }}^{N}f(T^{\Omega(n)}x)=\int_{X}fd\mu\]
\end{theorem}

Note that \Cref{cor:failure_full} and \Cref{thm:liwangwangyi_k_full} show that these types of averages also exhibit a heavy dependence on the assumptions of the system. For general ergodic systems and measurable functions, they satisfy the strong sweeping-out property, whereas they converge under the assumption of (totally) unique ergodicity for continuous functions.

Further examples of objects satisfying an Erd\H{o}s--Kac type, to which \Cref{thm:general_failure_of_pointwise} applies, can be found in \cite{Granville_Soundararajan_sieving_erdos_kac:2007} and \cite{MurtyMurtyPujahari:2023}.

We now focus on \Cref{question:question_5.2}. To provide an affirmative answer, it suffices to show that the averaging scheme $((\{1, \dots, N\}^2)_{N\in \N}, \tau)$, with $\tau(m,n) = \Omega(m^2+n^2)$, satisfies an Erd\H{o}s--Kac-type theorem.
In order to do this, we first apply \Cref{thm:general_erdos_kac_polynomials} to the polynomial $g(x,y)=x^2+y^2$. Note that this polynomial is irreducible and it has only one zero. It also has the symmetries $g(x,y)=g(-x,y)=g(x,-y)=g(-x,-y)$ (which implies that its images are the same in the four quadrants of $\Z^2$). Thus, we obtain: 

\begin{lemma} For $\gamma\in \R$, we have \label{cor:erdos_kac_small_omega}
\[ \lim_{N \to \infty} \frac{\# \left\{ 1 \leq m, n \leq N : \frac{\omega(m^2 + n^2) - \log_2(N)}{\sqrt{\log_2(N)}} \leq \gamma \right\}}{N^2} = G(\gamma). 
\]\end{lemma}
It is common, but not always true, that when one asymptotic result holds for $\omega(\cdot)$, it also holds for $\Omega(\cdot)$ and vice versa. In our particular case, we can also show with some work that \Cref{cor:erdos_kac_small_omega} indeed holds with $\Omega(\cdot)$ in place of $\omega(\cdot)$. 
From now on, for a Gaussian integer $z \in \mathbb{Z}[i]$, we write $\mathcal{N}(z)$ for its (squared) Euclidean norm. This is consistent with the ideal norm, since for the principal ideal $(z)$ we have $
\mathcal{N}((z)) =z\overline{z}$.
\begin{lemma}
For $\gamma\in \R$, we have \label{cor:erdos_kac_big_omega}
\[ \lim_{N\to \infty} \frac{ \#\left\{1\leq m,n\leq N: \frac{\Omega(m^2+n^2)-\log_2(N)}{\sqrt{\log_2(N)}}\leq \gamma\right\} }{N^2}=G(\gamma). \]
\end{lemma}
\begin{proof} For $1\leq m,n \leq N$, we write
\[\frac{\Omega(m^2+n^2)-\log_2(N)}{\sqrt{\log_2(N)}}=\frac{\omega(m^2+n^2)-\log_2(N)}{\sqrt{\log_2(N)}}+\frac{\Omega(m^2+n^2)-\omega(m^2+n^2)}{\sqrt{\log_2(N)}}\]
In view of \Cref{thm:slutsky} and \Cref{cor:erdos_kac_small_omega}, to get the conclusion, it suffices to show that the random variables $X_N(m,n)=\frac{\Omega(m^2+n^2)-\omega(m^2+n^2)}{\sqrt{\log_2(N)}}$, defined on $\{1,\ldots,N\}^2$ with the uniform probability measure satisfy 
\begin{equation}
\lim_{N\to \infty} \mathbb{E}(|X_N|)=0.
\label{eq:expected_value}
\end{equation}
Thus, we have to study \[\mathbb{E}\left( \frac{\Omega(m^2+n^2)-\omega(m^2+n^2)}{\sqrt{\log_2(N)}}\right).\]
We can rewrite it as follows, where $p^k\parallel m^2+n^2$ means $p^k| m^2+n^2$ but $p^{k+1}\nmid m^2+n^2$.
\begin{equation}\label{eq:rewriting_expected}
\mathbb{E}\left(\frac{\Omega(m^2+n^2)-\omega(m^2+n^2)}{\sqrt{\log_2(N)}}\right)=\frac{1}{\sqrt{\log_2(N)} N^2} \sum_{p\in \P} \sum_{k\geq 2}k\sum_{\substack{1\le m,n\le N\\ p^k \parallel m^2+n^2}} 1\end{equation}
Let us denote the innermost sum as $C(p,k,N) := \sum_{\substack{1\le m,n\le N\\ p^k \parallel m^2+n^2}} 1$. $ C(p,k,N) = O(N^2 \frac{k}{p^{k}}) $, where the implied constant is independent of $p$, $k$, and $N$. Before proving the bound, we show that it is sufficient to establish our result. Indeed, by \eqref{eq:rewriting_expected}, we have 
\[\mathbb{E}\left(\frac{\Omega(m^2+n^2)-\omega(m^2+n^2)}{\sqrt{\log_2(N)}}\right)\leq \frac{1}{\sqrt{\log_2(N)}} O(\sum_{p \in \P} \sum_{k\geq 2 } \frac{k^2}{p^{k}}).\]
Thus, it suffices to show that $\sum_{p \in \P} \sum_{k\geq 2 } \frac{k^2}{p^{k}}$ is uniformly bounded to conclude \eqref{eq:expected_value}. Note that for a fixed prime $p\in \P$, the inner sum is bounded by \[\sum_{k\geq 2 } \frac{k^2}{p^{k}}\leq \frac{C}{p^2},\]
where $C$ is a uniform constant that holds for all primes. Thus, we have
\[\sum_{p \in \P} \sum_{k\geq 2 } \frac{k^2}{p^{k}}\leq\sum_{p\in \P}\frac{C}{p^2}\leq C\sum_{n\in \N,n\geq 2} \frac{1}{n^2}<\infty,\]
which yields the conclusion.

Now, we proceed with the proof of the claim. We interpret $m^2+n^2$ as the norm of the Gaussian integer $m+ni$. To continue, we analyze three distinct cases based on the prime factorization of $p \in \mathbb{P}$ as a Gaussian integer.
\begin{itemize}
 \item Case $p=2$, in this case $2=-i(1+i)^2$, where $(1+i)$ is a Gaussian prime, this is the ramified case.
 \item Case $p\equiv 1 \pmod{4}$, in this case $p=\pi \cdot \bar{\pi}$, where both $\pi$ and $\bar{\pi}$ are relatively prime Gaussian primes, this is the split case. \item Case $p\equiv 3 \pmod{4}$, in this case $p$ is a Gaussian prime, this is the inert case.
 \end{itemize}
 We prove the claim only in the split case, since the other cases are analogous, and the split case yields the dominant asymptotic term.

Let $p$ be a prime such that $p\equiv 1 \pmod{4}$. Then $p=\pi \cdot \bar{\pi}$, where $\pi$ and $\bar{\pi}$ are coprime.
We write the prime factorization of $ m + ni = \pi^{a} \bar{\pi}^{b} c $, where $c$ is coprime to both $\pi$ and $\bar{\pi} $. From this, we obtain $ m^{2} + n^{2} = p^{a+b} \mathcal{N}(c)$.
The condition $p^k\parallel m^2+n^2$ and the coprimality of $c$ with $\pi$ and $\bar{\pi}$, imply that $a+b=k$. Since $0\leq a, b\leq k$, we obtain $k+1$ cases. For each of these cases, as $m $ and $n $ range from $ 1 $ to $ N $ we find that for every $(m+ni)$ such that $ p^k \parallel m^2+n^2 $ there is a unique $c$, such that $m+ni=\pi^{a}\bar{\pi}^bc$ and $\mathcal{N}(c)\leq \frac{m^2+n^2}{p^k}\leq\frac{2N^2}{p^k}$. Therefore there are at most $\#D_N$ solutions for each case, where $ D_{N} = \{z\in \mathbb{Z}[i] : \mathcal{N}(z)\leq \frac{2N^2}{p^k} \} $. By combining all the cases, we obtain
\[C(p,k,N)\leq (k+1)\#D_{N}\leq 2k \#D_N.\]
Now, by Gauss's progress on the circle problem or alternatively by \Cref{thm:weber_theorem} in the case $\K=\Q[i]$, we have that $\#D_{N}=O(\frac{N^2}{p^k})$, proving the claim.
\end{proof}
Now, we can deduce the answer to \Cref{question:question_5.2} directly.
\begin{theorem}
Let $(X, \mathcal{X}, \mu, T)$ be non-atomic invertible ergodic m.p.s. Then for every $\varepsilon>0$, there is a measurable set $A$, satisfying $\mu(A)<\varepsilon$ and such that for $\mu$-a.e. $x\in X$
\[
\limsup_{N\to\infty}\frac{1}{N^2}\sum_{1\leq m, n \leq N} \one_{A}(T^{\Omega(m^2+n^2)}x)
= 1,
\]
and
\[
\liminf_{N\to\infty}\frac{1}{N^2}\sum_{1\leq m, n \leq N} \one_{A}(T^{\Omega(m^2+n^2)}x)
= 0.\]
\end{theorem}
\begin{proof}
For $N \in \mathbb{N}$, let $\Phi_N =\{1, \dots, N\}^2$. We define $\tau(m,n) = \Omega(m^2 + n^2)$. By \Cref{cor:erdos_kac_big_omega} the averaging scheme $((\Phi_N)_{N\in \N},\tau)$ satisfies the hypothesis of \Cref{thm:general_failure_of_pointwise}, and we can conclude.
\end{proof}
Finally, we present the proof of \Cref{thm:general_failure_of_pointwise}.
\begin{proof}[Proof of \Cref{thm:general_failure_of_pointwise}] Let $\varepsilon>0$.
Fix an averaging scheme $((\Phi_N)_{N\in\N}, \tau)$ that satisfies an Erd\H{o}s--Kac type theorem. Then, for every $\gamma>0$, we have
\[\lim_{N\to \infty} \frac{\#\{k\in \Phi_N:-\gamma\leq\frac{\tau(k)-c\log_2(N)}{\sqrt{c\log_2(N)}}\leq\gamma \}}{\#\Phi_N}=G(\gamma)-G(-\gamma).\]
Since $\lim_{\gamma\to \infty}G(\gamma)-G(-\gamma)=1$, there is a $\gamma_1$, which we can take as an integer, and a large $N_1$, such that for every $N\geq N_1$, we have
\begin{equation}
\frac{\#\{k\in \Phi_N:-\gamma_1\leq\frac{\tau(k)-c\log_2(N)}{\sqrt{c\log_2(N)}}\leq \gamma_1\}}{\#\Phi_N
} \geq 1-\varepsilon. \label{eq:hardy_ramanujan_averaging_scheme}
\end{equation}
Rewriting \eqref{eq:hardy_ramanujan_averaging_scheme}, we obtain
{\small
\begin{equation} \label{eq:hardy_ramanujan_averaging_scheme_small} \frac{\#\{k\in \Phi_N: c \log_2(N) -\gamma_1{\sqrt{c\log_2(N)}} \leq \tau(k)\leq c\log_2(N) +\gamma_1{\sqrt{c\log_2(N)}}\}}{\#\Phi_N
} \geq 1-\varepsilon.\end{equation}}
We want to apply \Cref{thm:transference_principle}. Thus, we have to show that for our fixed $\varepsilon>0$ and any given $M>0$, there is a finite set $E\subseteq \Z$ and a $N_2\geq N_0$, such that
\begin{equation}\label{eq:maximal_ineq_proof}
\#\{x\in \Z: \sup_{N_0\leq N\leq N_2}\frac{1}{\#\Phi_N} \sum_{k\in \Phi_N}\one_{E}(x+\tau(k))\geq 1-\varepsilon\} \geq M \# E.
\end{equation}
Now, we follow closely the proof of \cite[Corollary 1.6]{Mondal2023}. 
We choose a large integer $K$, such that 
\[\lfloor{e^{e^{\frac{K^2}{c}}}}\rfloor\geq \max \{N_1,N_0\}.\]
Let us take $E=[-8\gamma_1 K, 8\gamma_1K]\cap\Z$. For $l \in \{0,...,K\}$, denote $M_l=\lfloor e^{e^{\frac{(K+l)^2}{c}}}\rfloor$. Let $x\in [-(2K)^2-2\gamma_1K, -K^2+2\gamma_1K]\cap\Z$. Then there exists an $l\in \{0,...,K\}$, such that $x\in [-(K+l)^2-2\gamma_1K, -(K+l)^2 +2\gamma_1K]\cap \Z$. Applying \Cref{eq:hardy_ramanujan_averaging_scheme_small} with $N=M_l$, we obtain
\[\frac{\#\{k\in\Phi_{M_l}:x+\tau(k) \in E\}}{\#\Phi_{M_l}} \geq 1-\varepsilon,\]
which is equivalent to
\[\frac{1}{\#\Phi_{M_l}}\sum_{k\in \Phi_{M_l}}\one_E(x+\tau(k))\geq 1-\varepsilon .\]
In other words, we have shown that for $N_2=M_K$, the following holds
\begin{equation}\label{ineq:proof_general}
\begin{aligned}
\#\Big\{x\in \Z :\;& 
\sup_{N_0\leq N\leq N_2}\frac{1}{\#\Phi_N}
\sum_{k\in \Phi_N}\one_{E}(x+\tau(k))
\geq 1-\varepsilon \Big\} \\
&\geq \#\big([-(2K)^2-2\gamma_1K,\,-K^2+2\gamma_1K]\cap\Z\big).
\end{aligned}
\end{equation}
Note that we can rewrite the right-hand side of \eqref{ineq:proof_general} as 
 \[(3K^2+4\gamma_1K+1)=(16K\gamma_1+1)g(K)=\#Eg(K),\]
 where $\lim_{K\to \infty}g(K)=\infty$. That is, if we choose a large $K$ such that $g(K)\geq M$ and that satisfies the previously imposed conditions, then \eqref{eq:maximal_ineq_proof} holds.
\end{proof}
\section{Convergence in uniquely ergodic systems and applications}\label{section:applications}
\subsection{Convergence in the topological setting}
To address pointwise convergence in uniquely ergodic systems as mentioned in \Cref{thm:dynamical_omega_number_fields}, we adopt an approach outlined in \cite[Remark 1.3]{Bergelson_Richter_dynamical_pnt:2022}, which was introduced to Bergelson and Richter by Kanigowski and Radzwi\l{}\l{} for the particular case where $\K=\Q$ and $c=\Omega_{\Q}$. It consists of proving the asymptotic shift-invariance of an average using the statistical properties of $\omega_{\K}$ and $\Omega_{\K}$ described in \Cref{section:background}, from which the pointwise result follows immediately.
\begin{theorem}\label{thm:teo_arimetico_fields}
Let $\mathbb{K} $ be a number field. If $ a: \mathbb{N}\cup\{0\} \to \mathbb{C} $ is a bounded sequence, then for either $c(\a)=\omega_{\K}(\a)$ or $c(\a)=\Omega_\K(\a)$, we have 
\begin{align*}
&\frac{1}{
\#\{ \mathfrak{a} \in G_{\mathbb{K}} : \mathcal{N}(\mathfrak{a}) \leq N \}}
\sum_{\substack{\mathfrak{a} \in G_{\mathbb{K}} \\ \mathcal{N}(\mathfrak{a}) \leq N}} 
a\big(c(\mathfrak{a}) + 1\big) \\
&= \frac{1}{\#\{ \mathfrak{a} \in G_{\mathbb{K}} : \mathcal{N}(\mathfrak{a}) \leq N \}}
\sum_{\substack{\mathfrak{a} \in G_{\mathbb{K}} \\ \mathcal{N}(\mathfrak{a}) \leq N}} 
a\big(c(\mathfrak{a})\big) + o(1)_{N \to \infty}.
\end{align*}
\end{theorem} 
\begin{remark}
In the case of $\Omega_{\K}$, this result was proven without appealing to any form of the prime ideal theorem by Burgin in \cite{Burgin_new_elementary_landau:2025}.
\end{remark}
\begin{proof}
We provide the proof only for $ c = \Omega_\K$, as the other case is identical.
For $N \in \mathbb{N}$ and $k \in \mathbb{N}$, recall that $N_k(N) = \#\{ \mathfrak{a} \in G_{\mathbb{K}} : \mathcal{N}(\mathfrak{a}) \le N, \; \Omega_{\mathbb{K}}(\mathfrak{a}) = k \}$,
 and define $w_N(k) := \frac{N_k(N)}{\# \{ \mathfrak{a} \in G_{\mathbb{K}} : \mathcal{N}(\mathfrak{a}) \le N \}}$.
We may write
\[
\frac{1}{\#\{ \mathfrak{a} \in G_{\mathbb{K}}: \mathcal{N}(\mathfrak{a}) \leq N \}}\sum_{\substack{\mathfrak{a} \in G_\K \\ \mathcal{N}(\mathfrak{a}) \le N}}a(\Omega_{\K}(\mathfrak{a})) = \sum_{k \in \mathbb{N}} w_N(k+1) a(k+1)+o(1)_{N\to \infty},
\]
and
\[
\frac{1}{\#\{ \mathfrak{a} \in G_{\mathbb{K}}: \mathcal{N}(\mathfrak{a}) \leq N \}}\sum_{\substack{\mathfrak{a} \in G_\K \\ \mathcal{N}(\mathfrak{a}) \le N}}a(\Omega_{\K}(\mathfrak{a}) + 1) = \sum_{k \in \mathbb{N}} w_N(k) a(k+1)+o(1)_{N\to \infty}.
\]
Thus, it suffices to show that 
 \[ \lim_{N\to \infty}\sum_{k \in \N} |a(k)| \cdot|w_{N}(k+1) -w_{N}(k)|=0 . \]
Since $a$ is bounded, it suffices to prove
\begin{equation}
\lim_{N\to \infty} \sum_{k \in \mathbb{N}} |w_{N}(k+1) - w_{N}(k)|=0.
\label{equation:serie_weighted}
\end{equation}
To this end, let $\varepsilon>0$. By the Hardy-Ramanujan theorem (\Cref{thm:hardy_ramanujan}), there exists a sufficiently large $C>0$ such that most of the weight of the $w_{N}(k)$ functions is supported on an interval of the form $ k\in I_{N,C}$ (see \eqref{definition:interval_n_c}), that is
\[\lim_{N\to \infty}\sum_{k\not \in I_{N,C}} w_N(k)<\varepsilon.\]
Then
\[ \sum_{k \in \N} |w_{N}(k+1) -w_{N}(k)|= \sum_{k \in I_{N,C}} |w_{N}(k+1) -w_{N}(k)|+o(1)_{N\to \infty}
 +O(\varepsilon).\]
Now, using \Cref{lemma:approximation_level_sets}, we have
\[w_N(k)=\frac{e^{-\frac{1}{2}(\frac{k-\log_2(N)}{\sqrt{\log_2(N)}})^2}}{\sqrt{2\pi \log_2(N)}}(1+\varepsilon_{k}(N)),\]
where the error terms satisfy $\lim_{N\to \infty} \sup_{k\in I_{N,C}} |\varepsilon_{k}(N)|=0$. By summing the error terms and using the fact that they tend to zero uniformly in $N$, we obtain that \eqref{equation:serie_weighted} equals
\[\sum_{k \in I_{N,C}} \Bigl|\frac{e^{-\frac{1}{2}(\frac{k+1-\log_2(N)}{\sqrt{\log_2(N)}})^2}}{\sqrt{2\pi \log_2(N)}} -\frac{e^{-\frac{1}{2}(\frac{k-\log_2(N)}{\sqrt{\log_2(N)}})^2}}{\sqrt{2\pi \log_2(N)}}\Bigr| +o(1)_{N\to \infty}+O(\varepsilon ).\]
Using the basic equality $(k + 1 - \log_2(N))^2 = (k - \log_2(N))^2 + 2(k - \log_2(N)) + 1,$
we can rewrite \eqref{equation:serie_weighted} as 
\[ \sum_{k \in I_{N,C}} \frac{e^{-\frac{1}{2}(\frac{k-\log_2(N)}{\sqrt{\log_2(N)}})^2}}{\sqrt{2\pi \log_2(N)}}|e^{\frac{\log_2(N)-k}{\log_2(N)}- \frac{1}{2\log_2(N)}} -1| +O(\varepsilon )+o(1)_{N\to \infty}
\]

Since for $k\in I_{N,C}$ we have $-C\log_2(N)^{-\frac{1}{2}}\leq \frac{\log_2(N)-k}{\log_2(N)}\leq C\log_2(N)^{-\frac{1}{2}},$
we obtain 
\[\sum_{k \in I_{N,C}} \frac{e^{-\frac{1}{2}(\frac{k-\log_2(N)}{\sqrt{\log_2(N)}})^2}}{\sqrt{2\pi \log_2(N)}}|e^{\frac{\log_2(N)-k}{\log_2(N)}- \frac{1}{2\log_2(N)}} -1|=o(1)_{N\to\infty}\cdot\sum_{k \in I_{N,C}}\frac{e^{-\frac{1}{2}(\frac{k-\log_2(N)}{\sqrt{\log_2(N)}})^2}}{\sqrt{2\pi \log_2(N)}}.\]
Given that
\[ \sum_{k \in I_{N,C}}\frac{e^{-\frac{1}{2}(\frac{k-\log_2(N)}{\sqrt{\log_2(N)}})^2}}{\sqrt{2\pi \log_2(N)}}=O(1),\]
we conclude that 
\[ \sum_{k \in I_{N,C}} \frac{e^{-\frac{1}{2}(\frac{k-\log_2(N)}{\sqrt{\log_2(N)}})^2}}{\sqrt{2\pi \log_2(N)}}|e^{\frac{\log_2(N)-k}{\log_2(N)}- \frac{1}{2\log_2(N)}} -1|=o(1)_{N\to \infty}.\]
Collecting all terms gives
\[\sum_{k \in \mathbb{N}} |w_{N}(k+1) - w_{N}(k)| = O(\varepsilon) + o(1)_{N\to \infty}.\]
Thus, we obtain
\begin{equation}
\limsup_{N \to \infty} \sum_{k \in \mathbb{N}} |w_{N}(k+1) - w_{N}(k)| = O(\varepsilon) \label{eq:weights_less_epsilon}.
\end{equation}
Since this holds for any $\varepsilon>0$, we get \eqref{equation:serie_weighted}, finishing the proof.
\end{proof}

\begin{remark}
Note that we may generalize \cref{thm:teo_arimetico_fields} to include sequences $c(\a)$ that exhibit the distribution described in \cref{lemma:approximation_level_sets}. Such sequences satisfy the Hardy–Ramanujan property (see the discussion at the beginning of \cref{section:failure_of_pointwise}). We currently do not know whether satisfying both the Erd\H{o}s–Kac and Hardy–Ramanujan properties is sufficient for \cref{thm:teo_arimetico_fields} to hold. 
\end{remark}

The proof of \Cref{thm:dynamical_omega_number_fields} follows from \Cref{thm:teo_arimetico_fields} by nowadays standard arguments. For completeness, we provide them below.
\begin{proof}[Proof of \Cref{thm:dynamical_omega_number_fields}] For a fixed $x\in X$, consider the empirical measures \[\mu_N=\frac{1}{\#\{\mathfrak{a}\in G_\K: \mathcal{N}(\mathfrak{a})\leq N\}}\sum_{\substack{\mathfrak{a} \in G_\K \\ \mathcal{N}(\mathfrak{a}) \le N}}\delta_{T^{c(\mathfrak{a})}x},\quad N\in \N\] where $c=\omega_\K$ or $\Omega_\K$. Our goal is to show that $\mu_N$ converges to $\mu$ in the weak-* topology. Due to unique ergodicity, it suffices to show that every accumulation point of $\mu_N$ is $T$-invariant. This amounts to showing that for any $f\in C(X)$, 
 \[\lim_{N\to\infty}
\frac{1}{\#\{\mathfrak{a}\in G_{\K} : \mathcal{N}(\mathfrak{a})\le N\}}
\sum_{\substack{\mathfrak{a}\in G_{\K}\\ \mathcal{N}(\mathfrak{a})\le N}}
\bigl(f(T^{c(\mathfrak{a})}x)-f(T^{c(\mathfrak{a})+1}x)\bigr)
=0.\]
This claim follows directly from \Cref{thm:teo_arimetico_fields} by taking $a_n=f(T^nx)$.
\end{proof}
\subsection{Norm convergence}
To deduce norm convergence, we use a folklore transference principle that allows us to pass from results for rotations on the circle to general measure-preserving systems. More precisely, we will employ a generalization of the following well-known statement.
\begin{proposition}\label{proposition:folklore_averages}
Let $(a_n)_{n \in \mathbb{N}}$ be a sequence of nonnegative integers such that, for every $\lambda \in \mathbb{S}^1 \setminus \{1\}$
\[
\lim_{N \to \infty} \frac{1}{N} \sum_{n = 1}^{N} \lambda^{a_n} = 0.
\]
Then, for every m.p.s. $(X, \mathcal{X}, \mu, T)$ and every $f \in L^{2}(X, \mu)$
\[
\lim_{N \to \infty} 
\Bigg\lVert 
\frac{1}{N} \sum_{n = 1}^{N} f(T^{a_n}x) - P(f)
\Bigg\rVert_{L^2(X, \mu)} = 0,
\]
where $P(f)$ denotes the orthogonal projection onto 
\[
I(T) := \{ g \in L^{2}(X, \mu) : g \circ T = g \}.
\]
Recall that if $(X, \mathcal{X}, \mu, T)$ is ergodic, then $P(f) = \int f \, d\mu$.
\end{proposition}
Within the framework of averaging schemes described in \Cref{subsection:strong_sweeping}, we can further extend \Cref{proposition:folklore_averages} as follows.
\begin{lemma}
\label{lemma:general_l_2_convergence}
Let $((\Phi_N)_{N\in \N},\tau)$ be an averaging scheme. If for every $\lambda\in \mathbb{S}^{1}\setminus\{1\}$ 
\[ \lim_{N\to \infty}\frac{1}{\#\Phi_N} \sum_{n \in \Phi_N} \lambda ^{\tau(n)}=0.\]
Then, for every m.p.s. $(X,\mathcal{X},\mu,T)$ and every $f\in L^{2}(X,\mu)$, we have
\[\lim_{N\to\infty}
\left\|\frac{1}{\#\Phi_N}\sum_{n \in \Phi_N} f\bigl(T^{\tau(n)}x\bigr) -P(f)\right\|_{L^2(X,\mu)}= 0.\]\end{lemma}
This follows by applying the Bochner--Herglotz spectral theorem (see, for example, \cite[Theorem C.9]{Einsiedler_Ward11}). Now, we have the tools to generalize \cite[Theorem 2.5]{Loyd_dynamical_omega:2023}, which corresponds to the special case $\K=\Q$ and $c=\Omega_{\Q}$.
\begin{theorem}\label{thm:omega_norm_convergence}
Let $\K$ be a number field and $(X, \mathcal{X}, \mu, T)$ a m.p.s. 
For any $f \in L^2(X,\mu)$ and for either $c(\mathfrak{a}) = \Omega_\K(\mathfrak{a})$ or $c(\mathfrak{a}) = \omega_\K(\mathfrak{a})$,
we have
\[\
\lim_{N \to \infty} \Bigg \lVert \frac{1}{\#\{ \mathfrak{a} \in G_\K : \mathcal{N}(\mathfrak{a}) \leq N \}} \sum_{\substack{\mathfrak{a} \in G_\K \\ \mathcal{N}(\mathfrak{a}) \leq N}} f\left(T^{c(\mathfrak{a})} x\right) - P(f) \Bigg \rVert_{L^2(X,\mu)} = 0.
\]
\end{theorem}
\begin{proof}
Using the fact that irrational rotations and rotations on $m$-points are uniquely ergodic, it follows from \Cref{thm:dynamical_omega_number_fields} that for $c = \omega_{\K}$ or $c = \Omega_\K$, and for any $\lambda\in \S^1 \setminus\{1\}$
\[\lim_{N\to \infty}\frac{1}{\#\{\mathfrak{a}\in G_\K: \mathcal{N}(\mathfrak{a})\leq N\}}\sum_{\substack{\mathfrak{a} \in G_\K \\ \mathcal{N}(\mathfrak{a}) \le N}}\lambda^{c(\a)}=0.\]
Then, by considering the averaging scheme given by $((\Phi_N)_{N\in \N},\tau)$, where $\Phi_N=\{ \mathfrak{a}\in G_\K: \mathcal{N}(\mathfrak{a})\leq N\}$, and $\tau=\omega_\K$, or $\tau=\Omega_\K$, we can conclude by applying \Cref{lemma:general_l_2_convergence}.
\end{proof}
By the same reasoning, we obtain the norm convergence of the averages as in \Cref{thm:ThmD}.
\begin{corollary}
Let $(X,\mathcal{X},\mu,T)$ be a m.p.s. and $f\in L^2(X,\mu)$. Then 
\[\lim_{N\to \infty} \Bigg \lVert \frac{1}{N^2}\sum_{1\leq m, n \leq N} f(T^{\Omega(m^2+n^2)}x)-P(f)\Bigg\rVert_{L^2(X,\mu)}=0.\]
\end{corollary}
\subsection{Applications to number theory}
Since we have ergodic theorems for uniquely ergodic systems, we can derive equidistribution results for specific sequences as corollaries. Here, we illustrate some applications.

First, we present a generalization of the Pillai--Selberg theorem \cite{Pillai_Mangoldt,Selberg_squarefree_1939} to number fields. This result has previously been established for $\Omega_\K$ by Burgin in \cite{Burgin_new_elementary_landau:2025}. The proof proceeds by analyzing exponential sums in $\mathbb{Z}_m$, and then interpreting them as ergodic averages in rotations on $m$-points. The desired result then follows directly from our understanding of these averages along $\omega_{\K}$ and $\Omega_\K$, as established in \Cref{thm:dynamical_omega_number_fields}.
\begin{corollary}
Let $\K$ be a number field. For every integer $m\geq 2$ and every $\ell \in \{0, \ldots, m-1\}$, we have
\begin{align*}
\lim_{N\to\infty}
\frac{\#\{\mathfrak a\in G_{\mathbb K} :
\mathcal{N}(\mathfrak a)\le N,\ 
\Omega_{\mathbb K}(\mathfrak a)\equiv \ell \pmod m\}}
{\#\{\mathfrak a\in G_{\mathbb K} :
\ \mathcal{N}(\mathfrak{a})\le N\}}
&= \frac{1}{m}, \\
\lim_{N\to\infty}
\frac{\#\{\mathfrak a\in G_{\mathbb K} :
 \mathcal{N}(\mathfrak a)\le N,\ 
\omega_{\mathbb K}(\mathfrak a)\equiv \ell \pmod m\}}
{\#\{\mathfrak a\in G_{\mathbb K} :
\mathcal{N}(\mathfrak a)\le N\}}
&= \frac{1}{m}.
\end{align*}
\end{corollary}
Using the same ideas, we present a generalization of a result of Erd\H{o}s-Delange \cite{Delange1958,Erdos1946}, which was also previously obtained by Burgin in \cite{Burgin_new_elementary_landau:2025} in the case of $\Omega_{\K}$. The proof follows from applying \Cref{thm:dynamical_omega_number_fields} to irrational rotations and utilizing Weyl’s equidistribution criterion. Recall that $\{x\} = x - \lfloor x \rfloor$ denotes the fractional part of a real number.
\begin{corollary}\label{cor:weyl_omega}
Let $\K$ be a number field. For any irrational number $\alpha$ and any interval $I \subseteq [0,1)$, we have 
\[\lim_{N\to \infty} \frac{1}{\#\{\a\in G_{\K}:\mathcal{N}(\a)\leq N \}} \sum_{\substack{\mathfrak{a} \in G_\K \\ \mathcal{N}(\mathfrak{a}) \le N}} \one_{I}(\{\alpha \Omega_{\K}(\a) \})=|I|,\]
and
\[\lim_{N\to \infty} \frac{1}{\#\{\a\in G_{\K}:\mathcal{N}(\a)\leq N \}} \sum_{\substack{\mathfrak{a} \in G_\K \\ \mathcal{N}(\mathfrak{a}) \le N}} \one_I(\{\alpha \omega_{\K}(\a) \})=|I|.\]
\end{corollary}
More generally, the same conclusion of \Cref{cor:weyl_omega} holds, if we replace the polynomial sequence $p(n)=n \alpha$ (where $\alpha$ is irrational) by any generalized polynomial $g(n)$ (see \cite{Bergelson_Leibman07} for the definitions) such that $(g(n))_{n\in \N}$ is uniformly distributed modulo $1$.
Recall that a sequence $(x_n)_{n\in\mathbb N}$ is said to be
\emph{uniformly distributed modulo $1$} if for every interval
$I\subset [0,1)$, $\lim_{N\to\infty}\frac{1}{N}\#\{1\le n\le N : \{x_n\}\in I\}= |I|$. To establish this extension, we require the following additional input.
\begin{theorem}\cite[Theorem A]{Bergelson_Leibman07}
Let $g:\mathbb{N}\cup \{0\}\to \mathbb{R}$ be a generalized polynomial. Then there exists a nilmanifold $(X,T)$, a Riemann-integrable function $f$ and $x\in X$ such that \label{thm:bergelson_leibman}
\[
\{g(n)\} = f(T^n x), \quad n \in \N\cup\{0\}.
\] 
\end{theorem}
By combining \Cref{thm:bergelson_leibman} with the fact that orbits in nilsystems are always uniquely ergodic (see, for example, \cite[Chapter 11]{Host_Kra_nilpotent_structures_ergodic_theory:2018}), and by applying Weyl’s equidistribution criterion, we may rephrase \cite[Corollary 1.6]{Bergelson_Richter_dynamical_pnt:2022} in terms of averaging schemes as follows.
\begin{theorem} \label{thm:equidistr_scheme}
Let $((\Phi_N)_{N\in\mathbb{N}},\tau)$ be an averaging scheme satisfying the
following condition: for every uniquely ergodic system $(X,T)$ with invariant measure
$\mu$, every $f\in C(X)$, and every $x\in X$,
\[   \lim_{N\to\infty}\frac{1}{\#\Phi_N}\sum_{n\in\Phi_N}
    f(T^{\tau(n)}x) = \int f\,d\mu.
\]
Let $g:\N\cup\{0\}\to \R$ be a generalized polynomial. Then, the following are equivalent:
\begin{enumerate}
    \item  $g(n)$ is uniformly distributed modulo $1$
    \item For every interval $I \subseteq [0,1)$, we have
    \[
\lim_{N\to\infty} \frac{1}{\#\Phi_N} \sum_{n \in \Phi_N} \one_I(\{g(\tau(n))\}) = |I|.
\]
\end{enumerate}
\end{theorem}
As a corollary of \Cref{thm:equidistr_scheme,thm:dynamical_omega_number_fields},  we derive the following result.
\begin{corollary}
Let $\K$ be a number field. For any polynomial $ g(n) : \mathbb{N}\cup \{0\} \to \mathbb{R} $, where $ g(n) $ has at least one irrational non-constant term, for any $I \subseteq [0, 1) $, we have 
\[\lim_{N\to \infty}\frac{1}{ \#\{\a\in G_\K:\mathcal{N}(\a)\leq N\}} \sum_{\substack{\mathfrak{a} \in G_\K \\ \mathcal{N}(\mathfrak{a}) \le N}}\one_{I}(\{ g(\omega_\K(\a)) \})=|I|.\]
\end{corollary}


\begin{thebibliography}{99}
\bibitem{Akcoglu_et_al:1996}
M.~Akcoglu, A.~Bellow, R.~L.~Jones, V.~Losert, K.~Reinhold-Larsson and M.~Wierdl,
\newblock The strong sweeping-out property for lacunary sequences, Riemann sums, convolution powers, and related matters,
\newblock {\em Ergodic Theory and Dynamical Systems} \textbf{16} (1996), 207--253.
\newblock \href{https://doi.org/10.1017/S0143385700008798}{doi:10.1017/S0143385700008798}.

\bibitem{Bergelson_Leibman07}
V.~Bergelson and A.~Leibman,
\newblock Distribution of values of bounded generalized polynomials,
\newblock {\em Acta Math.} {\bf 198} (2007), no.~2, 155--230.
\newblock \href{https://doi.org/10.1007/s11511-007-0015-y}{doi:10.1007/s11511-007-0015-y}.

\bibitem{Bergelson_Richter_dynamical_pnt:2022}
V.~Bergelson and F.~K.~Richter,
\newblock Dynamical generalizations of the prime number theorem and disjointness of additive and multiplicative semigroup actions,
\newblock {\em Duke Math. J.} {\bf 171} (2022), no.~15, 3133--3200.
\newblock \href{https://doi.org/10.1215/00127094-2022-0055}{doi:10.1215/00127094-2022-0055}.


\bibitem{Billingsley:1969}
P.~Billingsley,
\textit{On the central limit theorem for the prime divisor function},
Amer. Math. Monthly \textbf{76} (1969), no.~2, 132--139.
\newblock \href{https://doi.org/10.1080/00029890.1969.12000157}{doi:10.1080/00029890.1969.12000157}.

\bibitem{Kryloff_Bogoliouboff_theorie_measure_systeme_dyn:1937}
N.~Bogoliouboff and N.~Kryloff,
\newblock La th\'eorie g\'en\'erale de la mesure dans son application \`a l'\'etude des syst\`emes dynamiques de la m\'ecanique non lin\'eaire,
\newblock {\em Ann. of Math.} {\bf 38} (1937), no.~1, 65--113.



\bibitem{Burgin_new_elementary_landau:2025}
A.~Burgin,
\newblock A new elementary proof of Landau’s prime ideal theorem, and associated results,
\newblock {\em Res. Number Theory} {\bf 11} (2025), Art.~28.
\newblock \href{https://doi.org/10.1007/s40993-025-00617-x}{doi:10.1007/s40993-025-00617-x}.
\bibitem{Delange1958}
H.~Delange,
\newblock On some arithmetical functions,
\newblock {\em Illinois J. Math.} {\bf 2} (1958), 81--87.
\bibitem{Das_Kuo_Liu_distribution_generators_monoids:2026}
S.~Das, W.~Kuo and Y.~R.~Liu,
\newblock On the distribution of the total number of generators of $h$-free and $h$-full elements in an Abelian monoid,
\newblock {\em Czech Math. J.} (2026), 
\newblock \href{https://doi.org/10.21136/CMJ.2026.0365-25}{doi:10.21136/CMJ.2026.0365-25}.
\bibitem{Das_Kuo_Liu_erdos_kac_subsets_number_fields:2025}
S.~Das, W.~Kuo and Y.~R.~Liu,
\newblock A subset generalization of the Erd\H{o}s-Kac theorem over number fields with applications,
\newblock Preprint (2025), arXiv:2506.03215 [math.NT], \url{https://arxiv.org/abs/2506.03215}. 
\bibitem{Das_Kuo_Liu_distribution_generalization_erdos_kac:2026}
S.~Das, W.~Kuo and Y.-R.~Liu,
\newblock Generalization of Erd\H{o}s--Kac theorem with applications,
\newblock To appear in {\em Canad. J. Math.}
\bibitem{delJunco_Rosenblatt_counterex_erg_numb_th:1979}
A.~del Junco and J.~M. Rosenblatt,
\newblock {Counterexamples in ergodic theory and number theory},
\newblock \emph{Mathematische Annalen}, \textbf{245} (1979), no.~3, 185--197.
\bibitem{Donoso_Le_Moreira_Sun_Averages_multiplicative_Gaussians:2024}
S.~Donoso, A.~N.~Le, J.~Moreira, and W.~Sun,
\newblock Averages of completely multiplicative functions over the Gaussian integers -- a dynamical approach,
\newblock {\em Trans. Amer. Math. Soc.} {\bf 377} (2024), 7081--7115.
\newblock \href{https://doi.org/10.1090/tran/9184}{doi:10.1090/tran/9184}.
\bibitem{Einsiedler_Ward11}
M.~Einsiedler and T.~Ward,
\newblock {\em Ergodic Theory with a View Towards Number Theory},
\newblock Grad. Texts in Math., vol.~259, Springer, Berlin, 2011.
 \bibitem{ElBaz_loughran_sofos_normal:2022}
D.~El-Baz, D.~Loughran, and E.~Sofos,
\newblock Multivariate normal distribution for integral points on varieties,
\newblock {\em Transactions of the American Mathematical Society} \textbf{375} (2022), no.~5, 3089--3127.
\bibitem{Erdos1946}
P.~Erd\H{o}s,
\newblock On the distribution function of additive functions,
\newblock {\em Ann. of Math.} (2) {\bf 47} (1946), 1--20.
\newblock \href{https://doi.org/10.2307/1969031}{doi:10.2307/1969031}.
\bibitem{Erdos_integers_k_prime_factors:1948}
P.~Erd\H{o}s,
\newblock On the integers having exactly \(k\) prime factors,
\newblock {\em Ann. of Math.} (2) {\bf 49} (1948), 53--66.
\newblock \href{https://doi.org/10.2307/1969113}{doi:10.2307/1969113}.
\bibitem{Granville_Soundararajan_sieving_erdos_kac:2007}
A.~Granville and K.~Soundararajan, 
\textit{Sieving and the Erd\H{o}s--Kac theorem}. 
In: A.~Granville and Z.~Rudnick (eds), \textit{Equidistribution in Number Theory, An Introduction}. 
NATO Science Series, vol. 237. Springer, Dordrecht, 2007. 
\url{https://doi.org/10.1007/978-1-4020-5404-4_2}
\bibitem{LiWangWangYi:2025}
H.~Li, B.~Wang, C.~Wang, and S.~Yi,
\emph{Some ergodic theorems over squarefree numbers and squarefull numbers},
Acta Arith. \textbf{221} (2025), 117--140.
\href{https://doi.org/10.4064/aa240909-18-6}{doi:10.4064/aa240909-18-6}
\bibitem{Host_Kra_nilpotent_structures_ergodic_theory:2018}
B.~Host and B.~Kra,
\newblock {\em Nilpotent structures in ergodic theory},
\newblock Math. Surveys Monogr., vol.~236, Amer. Math. Soc., Providence, RI, 2018.
\newblock \href{https://doi.org/10.1090/surv/236}{doi:10.1090/surv/236}.

\bibitem{Liu_generalization_Erdos-Kac_appl:2004}
Y.-R.~Liu,
\newblock A generalization of the Erd\H{o}s–Kac theorem and its applications,
\newblock {\em Canad. Math. Bull.} {\bf 47} (2004), no.~4, 589--606.
\newblock \href{https://doi.org/10.4153/CMB-2004-057-4}{doi:10.4153/CMB-2004-057-4}.
\bibitem{Loyd_dynamical_omega:2023}
K.~Loyd,
\newblock A dynamical approach to the asymptotic behavior of the sequence~$\Omega(n)$,
\newblock {\em Ergodic Theory Dynam. Systems} {\bf 43} (2023), no.~11, 3685--3706.
\newblock \href{https://doi.org/10.1017/etds.2022.81}{doi:10.1017/etds.2022.81}.
\bibitem{loyd_mondal_omega:2025}
K.~Loyd, S.~Mondal,
\newblock Ergodic averages along sequences of slow growth,
\newblock {\em J. London Math. Soc. } {\bf 111} (2025), no.~3.
\newblock \href{https://doi.org/10.1112/jlms.70124}{doi:10.1112/jlms.70124}.
\bibitem{Marcus_Number_Fields:1977}
D.~A.~Marcus,
\newblock {\em Number Fields},
\newblock Universitext, Springer-Verlag, New York–Heidelberg, 1977.
\bibitem{Mondal2023}
S.~Mondal,
\emph{Behavior of ergodic averages along a subsequence and the grid method},
Ph.D. dissertation, University of Memphis, 2023.
\bibitem{MurtyMurtyPujahari:2023}
M.~Ram~Murty, V.~Kumar~Murty, and S.~Pujahari,
\emph{An all-purpose Erd{\H{o}}s–Kac theorem},
Math. Z. \textbf{305} (2023), no.~45.
\newblock \href{https://doi.org/10.1007/s00209-023-03370-y}{doi:10.1007/s00209-023-03370-y}.
\bibitem{MondalRoyWierdl2023}
S.~Mondal, M.~Roy, and M.~Wierdl,
\newblock Sublacunary sequences that are strong sweeping-out,
\newblock {\em New York Journal of Mathematics} \textbf{29} (2023), 1060--1074.
\bibitem{RamVanOrder}
M.~Ram~Murty and J.~Van~Order,
\newblock Counting integral ideals in a number field,
\textit{Expo. Math.} \textbf{25} (2007), 53--66.
\bibitem{Pillai_Mangoldt}
S.~Pillai,
\newblock Generalisation of a theorem of Mangoldt,
\newblock {\em Proc. Indian Acad. Sci. Sect.~A} {\bf 13} (1931), 329--332.
\bibitem{Rosenblatt_Wierdl_ptwise_erg_via_harmonic:1995}
J.~M.~Rosenblatt, M.~Wierdl,
\newblock Pointwise ergodic theorems via harmonic analysis,
\newblock in {\em Ergodic Theory and Its Connections with Harmonic Analysis (Alexandria, 1993)},
London Math. Soc. Lecture Note Ser., vol.~205, Cambridge Univ. Press, Cambridge, 1995, pp.~3--151.
\newblock \href{https://doi.org/10.1017/CBO9780511574818.002}{doi:10.1017/CBO9780511574818.002}.
\bibitem{Sathe1953I}
L.~G.~Sathé,
\newblock On a problem of Hardy on the distribution of integers having a given number of prime factors.~I,
\newblock {\em J. Indian Math. Soc. (N.S.)} {\bf 17} (1953), 63--82.
\bibitem{Sathe1953II}
L.~G.~Sathé,
\newblock On a problem of Hardy on the distribution of integers having a given number of prime factors.~II,
\newblock {\em J. Indian Math. Soc. (N.S.)} {\bf 17} (1953), 83--141.
\bibitem{Sathe1954III}
L.~G.~Sathé,
\newblock On a problem of Hardy on the distribution of integers having a given number of prime factors.~III,
\newblock {\em J. Indian Math. Soc. (N.S.)} {\bf 18} (1954), 27--42.
\bibitem{Sathe1954IV}
L.~G.~Sathé,
\newblock On a problem of Hardy on the distribution of integers having a given number of prime factors.~IV,
\newblock {\em J. Indian Math. Soc. (N.S.)} {\bf 18} (1954), 43--81.
\bibitem{Selberg1954}
A.~Selberg,
\newblock Note on a paper by L.~G.~Sathé,
\newblock {\em J. Indian Math. Soc. (N.S.)} {\bf 18} (1954), 83--87.
\bibitem{Selberg_squarefree_1939}
S.~Selberg,
\newblock Zur Theorie der quadratfreien Zahlen,
\newblock {\em Math. Z.} {\bf 44} (1939), 306--318.
\bibitem{Tenenbaum:2015}
G. Tenenbaum,
\textit{Introduction to Analytic and Probabilistic Number Theory},
3rd ed., Graduate Studies in Mathematics, vol. 163,
American Mathematical Society, Providence, RI, 2015.
\bibitem{Wu_sharpening_Selberg-Delange:1996}
J.~Wu,
\newblock A sharpening of effective formulas of Selberg–Delange type for some arithmetic functions on the semigroup $G_\K$,
\newblock {\em J. Number Theory} {\bf 59} (1996), 1--19.
\newblock \href{https://doi.org/10.1006/jnth.1996.0085}{doi:10.1006/jnth.1996.0085}.
\end{thebibliography}
\end{document}